\documentclass[11pt,reqno]{amsart}
\usepackage{amsmath}
\usepackage{amssymb, latexsym, amsfonts, amscd, amsthm, mathrsfs, enumerate, esint}
\usepackage[usenames,dvipsnames]{color}
\usepackage[all]{xy}
\usepackage{graphicx}
\usepackage{epsfig}
\usepackage[colorlinks=true,linkcolor=blue]{hyperref}
\usepackage{url}
\usepackage{bm}
\usepackage{stmaryrd}
\usepackage[letterpaper,hmargin=1.0in,vmargin=1.0in]{geometry}
\usepackage{tikz}
\usetikzlibrary{shapes,arrows}

\usepackage{amsmath}
\usepackage{amssymb,amsbsy,amsthm}
\usepackage{graphicx}

\usepackage{enumerate}
\usepackage{cases}

\usepackage{pdfsync}
\usepackage{hyperref}
\usepackage{wrapfig} 

\allowdisplaybreaks[1]

\numberwithin{equation}{section}

\DeclareMathOperator{\N}{\mathbb{N}} 
\DeclareMathOperator{\Z}{\mathbb{Z}} 

\newcommand{\p}{\mathbb{P}} 
\renewcommand{\P}{\mathbb{P}} 
\newcommand{\E}{\mathbb{E}} 
\newcommand{\e}{\varepsilon}
\newcommand{\sI}{\mathscr{I}}

\def\cA{{\mathcal A}}

\def\cC{{\mathcal C}}
\def\cD{{\mathcal D}}

\def\cS{{\mathcal S}}

\newtheorem{maintheorem}{Theorem}
\newtheorem{theorem}{Theorem}[section]
\newtheorem{lemma}[theorem]{Lemma}
\newtheorem{remark}[theorem]{Remark}
\newtheorem{proposition}[theorem]{Proposition}

\begin{document}

\title{Activated random walk on a Cycle}

\author{Riddhipratim Basu}
\address{Riddhipratim Basu, International Centre for Theoretical Sciences, Tata Institute of Fundamental Research, Bangalore, India}
\email{rbasu@icts.res.in}
\author{Shirshendu Ganguly}
\address{Shirshendu Ganguly, Departments of Statistics and  Mathematics, UC Berkeley, CA, USA}
\email{sganguly@berkeley.edu}
\author{Christopher Hoffman}
\address{Christopher Hoffman, Department of Mathematics, University of Washington,
Seattle, USA}
\email{hoffman@math.washington.edu}
\author{Jacob Richey}
\address{Jacob Richey, Department of Mathematics, University of Washington,
Seattle, USA}
\email{jfrichey@math.washington.edu }
\date{\today}

\maketitle


 \begin{abstract}
We consider Activated Random Walk (ARW), a particle system with mass conservation, on the cycle $\Z/n\Z$. One starts with a mass density $\mu>0$ of initially active particles, each of which performs a simple symmetric random walk at rate one and falls asleep at rate $\lambda>0.$ Sleepy particles become active on coming in contact with other active particles. There have been several recent results concerning fixation/non-fixation of the ARW dynamics on infinite systems depending on the parameters $\mu$ and $\lambda$. On the finite graph $\Z/n\Z$, unless there are more than $n$ particles, the process fixates (reaches an absorbing state) almost surely in finite time. In a first rigorous result for a finite system, establishing well known beliefs in the statistical physics literature, we show that the number of steps the process takes to fixate is linear in $n$ (up to poly-logarithmic terms), when the density is sufficiently low compared to the sleep rate, and exponential in $n$ when the sleep rate is sufficiently small compared to the density, reflecting the fixation/non-fixation phase transition in the corresponding infinite system as established in \cite{RS12}. 
% ...
 \end{abstract}

\section{Introduction} \label{sec:intro}

Consider the following interacting particle system on a one dimensional lattice. Given a configuration of particles, initially all active, the dynamics, which conserves the particles, proceeds as follows. Each active particle independently does a simple symmetric random walk at rate one in continuous time and falls asleep at rate $\lambda>0$. Each sleepy particle is awakened when an active particle occupies the same site. This model, known as Activated Random Walk (ARW), has attracted interest in non-equilibrium statistical mechanics as well as probability literature in recent years in connection with studying fixed energy sandpile models \cite{DVZ98, VDMZ98, VDMZ00, Dickman2001, VD05, DRS10, RS12}. The motivation of studying this model is two-fold. ARW can be regarded as a special case of driven diffusive epidemic process introduced by Spitzer in 1970s, and studied later in \cite{KS1,KS2,KS3,KS4}. ARW was also introduced in the physics literature as a more mathematically tractable approximation of the Stochastic Sandpile Model (SSM), and is one of the paradigm examples of the widely studied phenomenon of self-organized ciriticality (SOC) \cite{Manna90,Manna91, SOC2,SOC1}.

ARW is believed to manifest self organized criticality when run in a finite volume with carefully controlled driven diffusive dynamics. However, the rigorous study of ARW has so far been mostly restricted to the case of infinite volume limit where the counterpart of SOC is Absorbing state Phase Transition (APT) (although some recent results have called into question the exact relationship between these two notions \cite{fey, levinethreshold,houghsandpile}). Absorbing state Phase Transition was rigorously established for ARW on $\Z$ a few years ago in the fundamental work of Rolla and Sidoravicius \cite{RS12}. Let us briefly explain their result. Consider ARW started with initial configuration of particles coming from a product measure with density $\mu$; denote this process by ${\rm ARW}(\mu,\lambda)$. One would expect that for a fixed $\lambda$, if $\mu$ is very small, then all the particles will eventually fall asleep, whereas for large $\mu$ the activity would go on forever. Indeed, in \cite{RS12} it was shown that for each $\lambda>0$, there exists $\mu_c(\lambda)\in [\frac{\lambda}{\lambda+1},1]$ such that for $\mu<\mu_c$ the process ${\rm ARW}(\mu,\lambda)$ on $\Z$ fixates (i.e., the total number of jumps at origin is finite) almost surely, and for $\mu>\mu_c$ the process remains almost surely active forever. Observe that it is easy to understand heuristically why $\mu_c\leq 1$. If $\mu>1$, there are ``more particles than sites" and hence not all particles can eventually fall asleep \cite{Shellef10,GA10,RS12}. Complementing the results of \cite{RS12}, the first three authors showed in \cite{BGH15} that for any fixed $\mu>0$, the process almost surely does not fixate if $\lambda$ is sufficiently small; thus showing $\mu_c\to 0$ as $\lambda\to 0$, and answering a question from \cite{DRS10,RS12}. Subsequently, a statement similar to the latter was proven for transient Euclidean lattices in \cite{stauffer}, which also analyzed ARW dynamics on transitive graphs where the random walk is ballistic. However, to rigorously establish the critical or near-critical behaviour in these models seems far out of reach of the current mathematical techniques. 

The results in \cite{RS12,BGH15} are in the setting of ARW on the inifinite lattice $\Z$. Indeed, there has been a flurry of recent mathematically rigorous results on ARW following the breakthrough work \cite{RS12}, but most of them have been in the context of Euclidean lattices or other infinite graphs \cite{Shellef10,GA10, ST14, RT15,T14,stauffer}. From the point of view of understanding self-organized criticality, it is interesting to study this model on a finite lattice, with say a periodic boundary condition. On a finite graph, if the total number of particles is more than the number of vertices, then the process will continue for ever. If the total number of particles is at most the number of vertices, this is an absorbing Markov chain, so all the particles will almost surely fall asleep after a finite time. One would expect the absorbing state phase transition to be manifested in the finite process as a phase transition for the absorption time. For many finite systems of these type, it is generally believed that absorption time has three different scalings with the system size, polynomial (with different exponents) for the sub-critical and critical systems, and exponential for the super-critical system. Indeed a version of the above statement in the set up of \cite{RS12} and \cite{BGH15} respectively are the main results of this paper. In physics literature there are many non-rigorous and numerical results about the critical and near-critical scaling of this and related quantities for SSM and its many variants (see e.g.\ \cite{lubeck04} and references therein). However, as with the infinite system, rigorous analysis of the critical scaling behaviour remains a challenging problem.

\subsection{Main Results} \label{sec:results} 

Consider an $n$-cycle $\Z/n\Z$ with nearest neighbour edges. Fix $\lambda>0$ and $\mu\in (0,1)$. Consider the initial configuration with independent $\mbox{Ber}(\mu)$ many particles at each site. (We will denote the product Bernoulli measure by $\P^{\mu}$). Consider ARW started with this configuration with sleep rate $\lambda$; denote this process by ${\rm ARW}(\mu, \lambda)$. As mentioned before, this is an absorbing Markov chain and hence the process reaches the absorbing state of all sleepy particles (the set of all such configuration will be henceforth called the cemetery set  and written $\Delta$) after a finite time almost surely. Let $T_n(\mu, \lambda)$ denote the total number of attempts by any active particle to either jump or to try to sleep.  Note that the continuous time ARW dynamics can be coupled naturally to the following discrete time dynamics: at every positive integer time, pick an active particle uniformly at random. With probability $\frac{1}{2(1+\lambda)}$ each, the particle jumps to one of the neighbouring sites, and with probability $\frac{\lambda}{1+\lambda}$ it tries to fall asleep.  It is easy to see that under the natural coupling, $T_n(\mu, \lambda)$ is the absorption time of the latter dynamics.

Our first result shows that if the particle density $\mu$ is sufficiently small compared to $\lambda$ then $T_n(\mu,\lambda)$ is linear up to poly-logarithmic correction factors. 

\begin{maintheorem}
\label{sub}
Consider ${\rm ARW}(\mu,\lambda)$ on $\Z/n\Z$. For any $\lambda > 0$ and $\mu < \frac{\lambda}{1+\lambda}$, there exist positive constants $C_0,b$ depending on $\mu$ and  $\lambda$, such that for all large enough $n$,
\[
\p(T_n(\mu, \lambda) >  C_0 n\log^2n ) \le \frac{1}{n^b}.
\]
\end{maintheorem}

Observe that $\mu< \frac{\lambda}{\lambda+1}$ is precisely the regime in which \cite{RS12} showed fixation on $\Z$. It is also easy to observe that the bound is tight up to the polylogarithmic factor. To see this, observe that the expected total number of jumps a particle takes before trying to fall asleep is $\frac{1+\lambda}{\lambda}.$ Thus, if the number of particles is linear in $n$, the total number of jumps is also at least linear in $n$. 

In the heavily super-critical regime, i.e., when $\lambda$ is sufficiently small compared to $\mu$, we have the following complementary result (This result should be compared to \cite[Theorem 1]{BGH15}. ).

\begin{maintheorem}
\label{sup}
For any $0<\mu < 1$, there exists $\lambda_0 > 0$ and $c > 0$ such that for any $\lambda < \lambda_0$,  and all large enough $n,$
\[
\p(T_n(\mu, \lambda) <  e^{c n}) < e^{-cn}.
\]
\end{maintheorem}

Theorem \ref{sub} relies heavily on the uniformity of the locations of the particles in the initial configuration (note however, that the arguments in this paper do not depend on the specific nature of the Bernoulli distribution of the initial configuration). In fact, it can be shown that $T_n(\mu,\lambda)$ is at least of order $n^3$ when all the particles start at the origin.  
An exponential upper bound for $T_n(\mu,\lambda)$  is also relatively easy to establish.  See Remarks \ref{sub12} and  \ref{sup12} for further elaboration. 

 We list below the new contributions in this paper and relations to existing results: Although the linear to exponential phase transition for absorption time is widely expected in the statistical physics literature, to the best of our knowledge this is the first rigorous result establishing such a transition for some variant of fixed energy sandpile models.
We rely crucially on the recent progress \cite{RS12, BGH15} in understanding ARW on the infinite line.
However one needs certain new ideas to deal with the finite case. 
 Following the argument of \cite{RS12}, the main obstacle in showing fast fixation for low particle density is the wrapping around issue, that is to make sure the particles do not wrap around the cycle, and wake up already settled particles. We get around this by a block argument and a two-sided variant of the stabilizing algorithm in \cite{RS12}. 
 In the process of attaining  the quantitatively optimal result Theorem \ref{sub}, we encounter a particle system similar to internal erosion (see \cite{levinethreshold}). 
 
For the slow fixation part i.e., Theorem \ref{sup}, we first recall that the the argument from \cite{BGH15} essentially tells us that when we stabilize a certain density of particles on $\Z/n\Z$ until they hit $0$, for small enough sleep rate only a small fraction of the particles fall asleep. 
The key observation in this paper is  that the above step can be  
applied iteratively for exponentially many rounds. 
A naive application of the iteration scheme only allows the number of steps to be logarithmic in $n$ since a constant fraction of particles fall asleep in every round. However the finiteness of the environment allows us to recycle particles which fell asleep in earlier rounds once they get woken up in later rounds. 
To this end we strengthen the argument in \cite{BGH15} by showing that with exponentially small failure  probability all the particles that fall asleep get woken before too many particles are lost, thereby sustaining the process for exponentially many steps. 

In the next section we elaborate on the above ideas further.

\subsection{Sketch of the proofs}\label{sop}
A crucial property of many interacting particle systems  that serve as models of distributed networks is the Abelian Property. Informally it means that the final outcome of a certain probabilistic experiment does not depend on the order in which operations at different sites are performed. 

In the context of ARW, one exploits the Abelian Property via the Diaconis-Fulton representation, which is roughly the following: (see Section \ref{fvap}, for a more formal description). At every site in $\Z/n\Z$, we have a `stack of instructions' which is a sequence of  i.i.d. instructions to jump to one of the neighbours with probability $\frac{1}{2(1+\lambda)}$ each, or is a sleep instruction with probability $\frac{\lambda}{1+\lambda}.$
Given the above stacks, one way to run the process is: as long as there is some active particle, pick an arbitrary site $x\in \Z/n\Z$ with at least one active particle and use the first unused instruction from the stack to topple the site. That is, if the instruction was a jump instruction, then the particle jumps to a neighbouring site accordingly, and otherwise tries to fall asleep. 

The Abelian Property then states that the final configuration of particles after every particle has fallen asleep does not depend on the order in which the sites were toppled. Thus in this language $T_n(\mu,\lambda)$
is the total number of instructions across all the stacks used  until the end of the toppling process.

We also rely on the following heuristically plausible monotonicity property of the ARW dynamics: given a set of stack instructions, while toppling sites, if we ignore any sleep instruction, i.e., the configuration does not turn an active particle to a sleepy particle, even though the instruction is a sleep instruction, then the total number of topplings required to reach $\Delta$ can only increase, (see Lemma \ref{l:sleepmon} and the discussion preceding it, for formal definitions). 

\subsubsection{Sketch of the proof of Theorem \ref{sub}}
Given the above two properties, to prove Theorem \ref{sub} we will provide a toppling scheme which will end with a configuration in the cemetery set $\Delta.$ Our toppling procedure will ignore certain sleep instructions and hence by the monotonicity property, the total number of instructions used in the actual process in the Diaconis-Fulton representation is upper bounded by the number of instructions used in our scheme. 
The basic idea is to break the cycle $\Z/n\Z$ in to sub-intervals $I_1,I_2,\ldots$ of size $O(\log n)$. Our toppling scheme is then a combination of toppling schemes, one for each of the sub-intervals. The toppling scheme for $I_i$ is designed to stabilize particles inside $I_i$ for  $i=1,2,\ldots,\frac{n}{\log n}$. The toppling scheme in each of the intervals is a variant of the trap-setting procedure appearing in \cite{RS12}.  We prove that with very small failure probability (exponential in the size of the interval) the toppling procedure in the interval succeeds to stabilize everything. A union bound over all the intervals show that with high probability the procedure succeeds simultaneously for all the intervals, and hence stabilizes the system. Recall that the failure probability for each interval is exponential in the size of the interval which forces us to choose the  size of the intervals to be logarithmic (in $n$)  as otherwise the union bound over polynomially (in $n$) many such intervals will fail; this also explains the logarithmic correction term in the statement of Theorem \ref{sub}. 

We now briefly describe our toppling scheme. It consists of broadly two parts. 
\begin{enumerate}
\item Phase 1: Given the initial configuration, the first step is to gather particles which are initially located uniformly over $\Z/n\Z$, to a set of points we call $\mathbf{Sources}$. We will take this set to be $\{\frac{c_0}{2}\log n, \frac{3c_0}{2} \log n, \ldots\}$ for some carefully chosen value of $c_0,$ depending on the parameters $\mu$ and $\lambda$.
Thus we first ignore all the sleep instructions and allow the particles to do independent random walks till they hit an element of $\mathbf{Sources}.$
Large deviation estimates imply that with high probability the number of particles at each $\mathbf{Source}$ at the end of this process is roughly $c_0 \mu \log n$. Recall that we are using the monotonicity property mentioned above, and hence we can ignore certain sleep instructions.

\item Phase 2:  The intervals $I_1,I_2.\ldots$ inside which we run our toppling scheme are of size $c \log n,$ and centred at the sources. The proof proceeds by showing that there is a toppling procedure which carefully ignores certain sleep instructions allowing the particles to fall asleep only at certain well chosen `traps' which prevents interaction with other particles.  The remainder of the proof then shows that the above scheme succeeds with high probability.
\end{enumerate}
\begin{figure}[h] 
\centering
\begin{tabular}{cc}
\includegraphics[scale=.7]{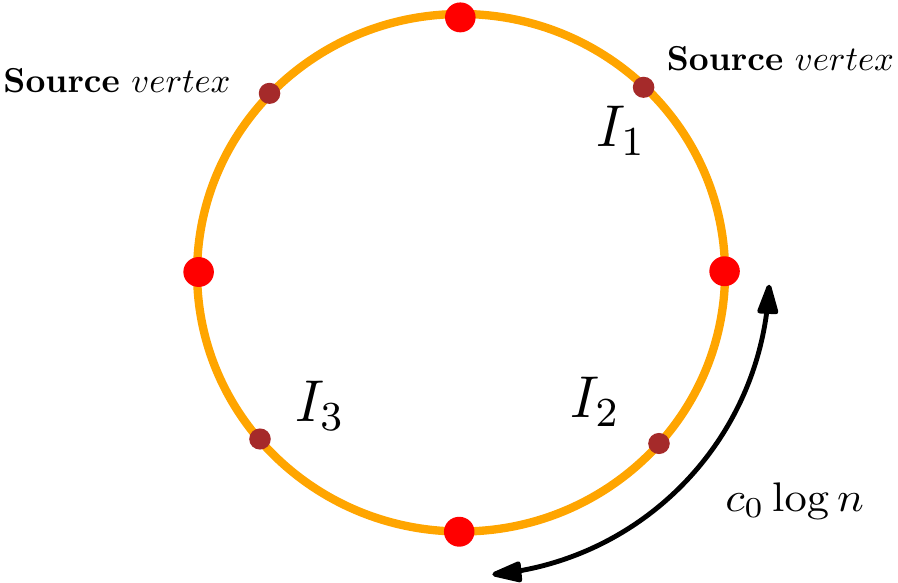} &\quad\quad\quad\quad\includegraphics[width=0.28\textwidth]{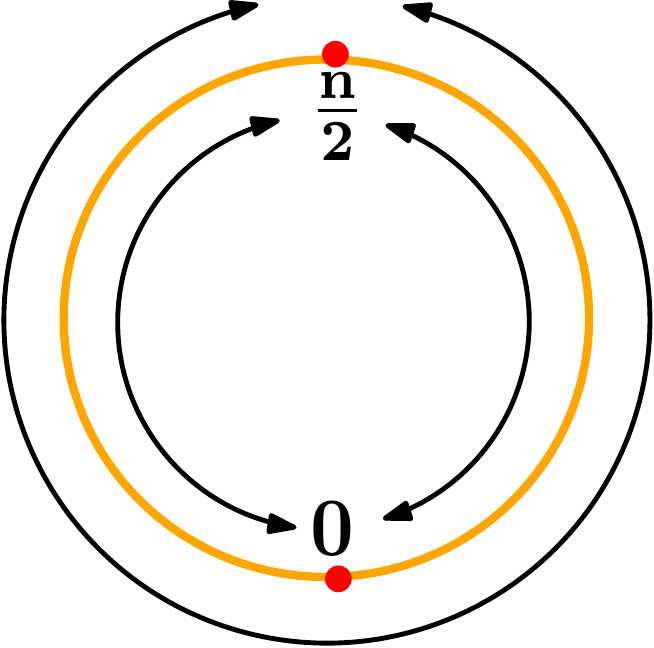} \\
(a) & \quad\quad\quad\quad(b)
\end{tabular}
\caption{ (a) The toppling scheme for fixation in the sub-critical regime: in the first step of the stabilization scheme, we  ignore sleep instructions to get every particle to a nearby source vertex. Particles at a particular source are then stabilized inside an interval of length $c_0\log n$ using a trap-setting procedure. (b) The toppling scheme for non-fixation in the supercritical regime: $0$ and $n/2$ are the north and south `poles' of the cycle. We run several rounds of the following stabilization loop where we first try to topple all particles located away from $0$ or $n/2$ until they fall asleep or hit $\{0, n/2\}$. Then we try to topple particles starting from $0$ until they hit $n/2$ or fall asleep, and afterwards do the same at $n/2$, namely topple all particles starting from $n/2$ until they hit $0$ or fall asleep. These three actions are repeated until all particles are asleep: our proof shows that this loop can be sustained for exponentially many steps (in $n$) with high probability. } 
\label{fig1}
\end{figure}

Even though the trap setting scheme is inspired from \cite{RS12}, the argument in the latter was particularly tailored to the ARW process on the infinite line and does not work on the cycle because of certain `wrapping around' issues. To circumvent this we introduce a  two sided version, which relying on estimates for random walk on the interval, can be shown to work in our setting.

\subsubsection{Sketch of the proof of Theorem \ref{sup}}
For the proof of Theorem \ref{sup} we rely on the non-fixation result from \cite{BGH15}. The technical core of that paper (see the proof of \cite[Lemma 18]{BGH15}) was to establish the following non-fixation phenomenon. Consider a sufficiently large interval $[0,r]$ with at least $\mu r $ many active particles, and stabilize the particle system inside the interior of the interval $[0,r]$, i.e., the particles are stopped upon hitting $\{0,r\}$. Then, given $\e>0,$ for any mass density $\mu$, and for all small enough $\lambda$, at the end of the stabilization procedure, the number of particles accumulating at $\{0,r\}$ is at least $(1-\e)$ fraction of the total number, with failure probability exponentially small in $r$. This was used in \cite{BGH15} to show infinite activity on the line, by considering a growing sequence of $r$ and then using the above statement to show that particles from arbitrary far away would hit the origin, thus implying  non-fixation. For a finite system, we cannot rely on an argument which uses particles arbitrarily far away. Instead we use the following `recycling particles' approach:
represent the cycle $\Z/n\Z$ as the interval $[-n/2, n/2]$ with endpoints identified. We run the following rounds of particle stabilization: 
\begin{enumerate}
\item We first treat the `poles' $0$ and $n/2$ as boundary points, and topple particles not at those sites until they fall asleep or land at one of the sites $0$ or $n/2$. By the arguments in \cite{BGH15}, a constant fraction of all particles will make it to either $0$ or $n/2$ with exponentially high probability.
\item We topple particles that ended up at $0$ in the previous stage until they fall asleep or hit $n/2$. The particles that did not start at $0$ do not move, but they can be woken up by other active particles during this stage. Using ideas from \cite[Lemma 18]{BGH15}, we show that most particles will make it to the boundary before falling asleep with exponentially high probability.
\item Just as the previous step, we only topple particles that started at $n/2$ at the end of the previous stage, running the dynamics until all such particles are asleep or at $0$. 
\end{enumerate}
Since all these steps keep most particles awake with exponentially high probability, we can run the loop exponentially many times. Moreover, since each loop takes at least one stack instruction to run, the Abelian Property implies that $T_n(\mu,\lambda)$ is at least exponential in $n$. 
\section*{Acknowledgements}
The authors thank Lionel Levine, Vladas Sidoravicius and Lorenzo Taggi for useful conversations, and Leonardo Rolla for comments on a previous version of the paper. Research of RB is partially supported by a Simons Junior Faculty Fellowship and a Ramanujan Fellowship from the Govt. of India. SG is supported by a Miller Research Fellowship.

\section{Abelian Property for Activated Random Walk}\label{fsetup}
For brevity, we will denote $\Z/n\Z$ by $\cC_n.$
We follow \cite{RS12} in formally describing the set up of ARW. To avoid unnecessary notational overhead we describe the bare minimum of the formalism necessary. We always work with ARW on $\cC_n$ for some fixed but large $n$, and the notation is adapted accordingly. In particular, for the remainder of this section, addition and subtraction will be considered modulo $n$ whenever appropriate and we shall not mention that explicitly every time.

For any time $t\ge 0$ and location $x \in \cC_n$, $\eta_t(x)$ denotes the state of the system at location $x$ at time $t$.  We write $\eta_t(x)= \rho$ if there is one sleepy particle 
at $x\in \cC_n$ at time $t$. If there is not a sleepy particle present we let $\eta_t(x) \in \N$
denote the number of particles at $x\in \cC_n$ at time $t$.
Then $\eta_t =\{\eta_t(x)\}_{x \in \cC_n}$ denotes the state of the system at time $t$.

We shall use two operations (called topplings) on the space of configurations. For $x\in\Z$ and $y=x\pm 1$, let $\tau_{x,y}(\eta)$ denote the configuration obtained by moving one of the active particles from $x$ to $y$. This operation will be called {\bf illegal} (for the configuration) if there are no active particles at $x$
and the system remains unchanged. 
Let $\tau_{x,\rho}(\eta)$ denote the configuration obtained from $\eta$ by making the solitary particle at $x$ fall asleep. Moreover if $x$ has more than one active particle, the sleep instruction has no effect, so $\tau_{x,\rho}(\eta)=\eta$. 
Again if there are no particles at $x$, this instruction is called illegal and the system is not changed.

Now we can formally define ARW as a finite state space continuous time Markov chain with transitions $\eta \rightarrow \tau_{x,y} \eta$ at rate $A(\eta_t (x))\frac{1}{2}\mathbf{1}_{y=x\pm 1},$ and $\eta\rightarrow \tau_{x,\rho}\eta$ at rate  $\lambda A(\eta_t (x))$
where $A(\eta(x))$ denote the number of active particles at site $x$ in configuration $\eta$.
Let $\P^{\nu}$ denote the law of the process started from an initial configuration distributed according to $\nu.$

\subsection{Diaconis-Fulton representation}\label{fvap}
We will now describe the Diaconis-Fulton representation of the ARW dynamics which will be convenient for our purposes. 
For an extensive discussion of the Diaconis-Fulton representation of ARW dynamics, the Abelian Property and its consequences, see \cite{RS12}. For completeness we recall the relevant results from \cite{RS12}, suitably adapted to the setting of a finite cycle. 

The Diaconis-Fulton representation \cite{DF91,Eriksson96} maps the ARW process to  sequence of instructions attached to the sites. The advantage of this representation is the  Abelian Property, which allows one to disregard the order in which different steps were performed in certain settings. We start by introducing a series of notations. Recall the operations $\tau_{x,y}$ and $\tau_{x,\rho} $ from above. Now consider the following array of random variables:
\begin{equation}\label{stack}
\mathscr{I}=
\begin{array}{ccccccc}
\ldots&\xi_{(-2,1)}&\xi_{(-1,1)}&\xi_{(0,1)}&\xi_{(1,1)}&\xi_{(2,1)}&\ldots\\
\ldots&\xi_{(-2,2)}&\xi_{(-1,2)}&\xi_{(0,2)}&\xi_{(1,2)}&\xi_{(2,2)}&\ldots\\
\vdots&\vdots&\vdots&\vdots&\vdots&\vdots&\vdots
\end{array}
\end{equation}
where $\xi_{(x,j)}$ are independent for any $x \in \cC_n$ and $j\in \N$ and moreover, 
\begin{equation}\label{law1}
\xi_{(x,j)}=\left\{ \begin{array}{cc}
\tau_{x,x-1} & \text{with probability~} \frac{1}{2(\lambda+1)}\\
\tau_{x,x+1} & \text{with probability~} \frac{1}{2(\lambda+1)}\\
\tau_{x,\rho} & \text{with probability~} \frac{\lambda}{\lambda+1}.
\end{array}
\right.
\end{equation} 
We will now show that using these instructions one can define a discrete time version of the ARW process.
In fact we can define many such versions.
But they will all have the same configuration when it is finally stabilized and the same set of instructions that have been implemented.

We call the $\xi_{(x,j)}$'s instructions at the site $x$ and the underlying product measure $\mathscr{P}$. 

Given a configuration $\eta$ at each discrete time step $t$, one can choose (arbitrarily) an unstable site $x$ and use the first unused element from the stack $\xi_{(x,\cdot)}$ and use it to perform the transition to a configuration $\eta'$ at time step $(t+1)$. As mentioned in Section \ref{sop}, we call such an operation ``toppling" at site $x$. 
We keep track of the number of topplings at every site. Let $\eta$ be the configuration after applying $h(x)$ many topplings at each site $x\in \cC_n$. Let us denote
\begin{equation}\label{odo12}
h:=(h(x): x \in \mathbb{Z}/n\Z)
\end{equation}
which we will call the odometer function. Let  $\Phi_x(\eta)$ denote the configuration obtained by toppling the site $x$ next, i.e., we apply the instruction $\xi_{x,h(x)+1}$, and also increase $h$ at $x$ by one ($h$ at other sites does not change). We say $\Phi_x$ is legal for $\eta$ if $x$ is unstable in $\eta$. For any sequence ${\bf{\alpha}}=(x_1,x_2,\ldots, x_k)$ we define the sequence of topplings at $x_1,$ followed by $x_2$ and so on through until $x_k$ by $\Phi_{\bf{\alpha}},$ i.e.\ $\Phi_{\bf{\alpha}}=\Phi_{x_k}\ldots \Phi_{x_1}$. We now say that ${\bf{\alpha}}$ is a {\bf legal sequence} for initial configuration $\eta$ if $\Phi_{x_i}$ is legal for $\Phi_{x_{i-1}}\ldots \Phi_{x_1}(\eta)$ for all $i=1,\ldots k.$
We abuse notation a little to denote by $h_{\alpha}$ the odometer function after performing the sequence of toppling given by $\alpha,$ i.e.\ for any $x \in \cC_n$, 
\begin{equation}\label{odo2}
h_{\alpha}(x)=\sum_{i=1}^{k}\mathbf{1}(\alpha_i=x).
\end{equation} 

Given the above preparation we can now formally state the Abelian Property, which says that given two sequence of legal topplings that result in the same odometer function (see \eqref{odo2}), the final configuration is the same in both the cases, i.e.\ the order in which topplings are performed does not matter. 
\begin{lemma}(Abelian Property, \cite[Lemma 2]{RS12})
\label{abp} 
Given any two legal sequence of topplings $\alpha$ and $\alpha'$ such that $h_{\alpha}=h_{\alpha'}$, then $$\Phi_{\alpha}(\eta)=\Phi_{\alpha'}(\eta).$$
\end{lemma}
The next lemma is a consequence of the Abelian Property. It shows that any legal sequence of topplings must occur in any stabilizing sequence (i.e., a legal sequence which leads to all stable sites).

\begin{lemma}(Least Action Principle, \cite[Lemma 1]{RS12})
\label{lap1}
Let $\alpha, \alpha'$ be two legal sequences of topplings such $\alpha$ stabilizes $\eta$, then $h_{\alpha'}\le h_{\alpha},$ i.e.\ all the topplings in $\alpha'$ are also needed in $\alpha.$ 
\end{lemma}

Observe that Lemma \ref{lap1} immediately implies that any two stabilizing sequence must lead to the same odometer function, and in turn by Lemma \ref{abp} this implies that the final configuration after stabilization is also independent of the stabilizing sequence. This will imply that for any stabilizing sequence $\alpha$, for an initial configuration of product $\mbox{Ber}(\mu)$ many particles, we have 
$$T_n(\mu,\lambda)=\sum_{x\in \cC_n} h_{\alpha}(x).$$
This will be the fundamental tool used in our proofs. Finally we need another lemma to compare the stabilizing sequences which formalizes the intuitively plausible statement: for a stabilizing sequence where we ignore some of the sleep instructions, the total number of jumps is larger than if we hadn't ignored those sleep instructions. 
Formally we need to introduce a new notation to define precisely the meaning of ignoring a sleep instruction. Recall the stack of instructions $\sI$ from \eqref{stack} and the action of the instructions $\tau_{x,\rho} $ and $\tau_{x,x\pm1}$ on the particle configuration $\eta.$ Let us introduce a null instruction $\frak{n}$ 
which acts on a configuration to create no change, i.e., $\frak{n} \eta=\eta$ for all $\eta.$
Now given a stack of instructions, $\sI=\{\xi_{(\cdot,\cdot)}\},$ as in \eqref{stack}, let $\sI'=\{\xi'_{(\cdot,\cdot)}\},$ be another stack, with the property that for each $(x,j)$ either , $\xi_{(x,j)}=\xi'_{(x,j)},$ or $\xi_{(x,j)}=\tau_{x,\rho},$ and $\xi'_{(x,j)}=\frak{n}.$
Thus informally $\sI'$ is any set of toppling instructions obtained from $\sI$ by ignoring certain sleep instructions.
\begin{lemma}\cite[Lemma 5]{RS12} Given any $\sI,\sI'$ as above, and any initial configuration $\eta,$ let $\alpha$ and $\alpha'$ be two legal toppling sequences stabilizing $\eta,$ using instructions from $\sI$ and $\sI',$ respectively. Let $h_{\alpha}(\cdot)$ and $h_{\alpha'}(\cdot)$ be the respective odometer functions. Then $h_{\alpha'}(x)\ge h_{\alpha}(x)$ for every $x\in \cC_n.$ 
\label{l:sleepmon}
\end{lemma}

Often our argument will be based on running the ARW dynamics on certain sub-intervals of $\cC_n,$ and hence to be completely formal one needs to introduce stacks corresponding to such intervals. However to avoid introducing additional notation, we will identify the latter in the natural way with the corresponding subset of stacks of $\sI.$

\section{Fast fixation under low density}  

We prove Theorem \ref{sub} in this section. By the discussion at the end of last section and Lemma \ref{l:sleepmon}, we shall provide an algorithm for toppling which will ignore some sleep instruction and which stabilizes all sites in $\cC_n$ within $O(n \log^2 n)$ many topplings with high probability. (Note that the probability here is over both the random initial configuration and also the random stack of instructions.) We shall formalize the sketch provided in Section \ref{sop} to build the toppling procedure. 

Let $\eta$ be a initial configuration of particles on $\cC_n$ distributed according to law $\P^{\mu}$. Also fix a realization $\mathscr{I}$ of the stack of instructions. For the remainder of this section, we shall always talk about toppling the configuration $\eta$ sequentially using instructions from $\mathscr{I}$.

\subsection{Random Walk estimates}
As explained before, the first phase will be to topple any unstable site that is not a $\mathbf{Source}$ and ignore any sleep instructions encountered in the process. So at the end of this phase of the toppling all the particles will be herded at the $\mathbf{Source}$ vertices. Let $\eta^{(1)}$ denote the configuration at the end of this phase, which is supported on $\mathbf{Source}$ vertices. Because we ignore the sleep instructions, this procedure is the same as letting all the particles in $\eta$ be independent simple symmetric random walks stopped at hitting one of the sources. Thus we shall need a couple of basic random walk estimates to estimate the distribution $\eta^{(1)}$, as well as the total number of jumps to reach that configuration.

Recall that the $i^{th}$ source is located at the vertex $z_i:=(i-\frac{1}{2})c_0\log n$. Let $\eta^{(1)}_i$ denote the number of particles $\eta^{(1)}$ has in this vertex. Also recall that to reduce notation we assumed that $n$ is an integer multiple of $c_0\log n$, which is also an integer.  Let $K=\frac{n}{c_0\log n}$ be the number of sources. In the general case we can take all the intervals to  be $\lfloor c_0\log n\rfloor$ except possibly one which has length between $\lfloor c_0\log n\rfloor$ and  $2 \lfloor c_0\log n\rfloor.$

The following lemma is our first random walk estimate.

\begin{lemma}
\label{l:rw1}
For each $\e>0$, there exists $a>0$ such that for all  large enough $c_0$ and $n$, 
\begin{align*}
\P\left(\sup_{1\le i\le K} |\eta^{(1)}_{i}-\mu c_0\log n| \ge \e c_0 \log n\right) \le e^{-a c_0 \log n}.
\end{align*}
\end{lemma}

\begin{proof} 
For this proof we shall forget about $\mathscr{I}$ and the toppling procedure, and treat $\eta^{(1)}$ as the configuration obtained from letting all the particles of $\eta$ perform independent simple symmetric random walks stopped at hitting any of the source vertices (thus particles initially located at source vertices do not move at all). We first recall a standard concentration inequality for sums of independent but not necessarily identically distributed Bernoulli variables that we will use later in the proof. 
Let $X_1,\ldots ,X_k$ be independent Bernoulli variables with means $p_1,\ldots,p_k$ respectively. Let $\nu=p_1+\ldots+p_k.$
Then 
\begin{align}
\label{azuma}
\P\left(\bigg|\sum_{i}^k X_i-\nu \bigg|\ge \delta \nu\right)\le e^{-\frac{\delta^2 \nu^2}{k}.}
\end{align}

Let us consider the first source $z_1$. Clearly, any particle that ended up at $z_1$ in $\eta^{(1)}$ must have been located at some $j\in V_1= (-\frac{c_0}{2}\log n, \frac{3c_0}{2}\log n)$. Let $Z_j$ denote the indicator of the event that there was a particle at $j$ in $\eta$ and that ended up at $z_1$ in $\eta^{(1)}$. Clearly $$\eta^{(1)}_{i}=\sum_{j\in V_1} Z_j.$$ Observe that a standard Gambler's ruin calculation yields that the probability that a random walk started at $j\in V_1$ would reach $z_1$ before reaching either $z_0$ or $z_2$ is $g(j)=1-\frac{|j-\frac{c_0}{2}\log n|}{c_0 \log n}$. It follows that $Z_j$'s are independent Bernoulli variables with mean $\mu g(j)$. Observe that $$\left |\sum_{j\in V_1} g(j)-c_0\log n \right |\leq 1,$$ and hence using \eqref{azuma} we get that for each $\e>0$, and for all $n$ sufficiently large
\begin{equation}
\label{e:Z1}
\P\left( |\eta^{(1)}_{1}-\mu c_0\log n| \ge \e c_0 \log n \right) \le e^{-2a c_0 \log n}
\end{equation}
for some $a$ depending on $\mu, \e$ (but not on $c_0$).
By the rotational symmetry of $\cC_n$ and of the law of the initial configuration $\mathbb{P}^\mu$, we have the same bound for all $\eta^{(1)}_{i}$ for all $1\le i \le K$. Now since the number of source vertices is less than $n$,  using \eqref{e:Z1} and a union bound over all source vertices, we get 
$$\P\left (\sup_{1\le i\le K} |\eta^{(1)}_{i}-\mu c_0\log n| \ge \e c_0 \log n \right ) \le n e^{-2a c_0 \log n} \leq e^{-a c_0 \log n}$$
where we have chosen $c_0$ sufficiently large (depending on $\mu, \e$) so that the last inequality holds.  This completes the proof of the lemma. 
\end{proof}

Recall the basic setting of toppling sites using instructions from stack  $\mathscr{I}$. The next lemma will prove that  the total number of instructions explored until the end of phase one is at most order $n\log ^2 n$ with high probability. 

\begin{lemma}
\label{l:nstep1}
Let $T^{(1)}(\mu,\lambda)$ denote the total number of instructions that have been explored until the end of phase one. Then there exists $C_2,\theta >0$ such that 
$$ \P\left( T^{(1)}(\mu,\lambda)> C_2 n\log^2 n \right)\leq n^{-\theta}.$$
\end{lemma}

For this step we shall need a concentration result for sums of geometric random variables. Although the result we need at this step is pretty standard, we shall need a more complicated variant later on, and we also need a concentration for some of exponential random variables for a later part of the argument. For convenience we quote, at this point, the following result from \cite{J14} which shall cover all our needs.

\begin{lemma}
\label{lem2} 
The following concentration results hold:
\begin{enumerate} \item[(i)] Fix $p \in (0,1)$, and let $Y_1, Y_2, \ldots$ be i.i.d.\ geometric random variables with parameter $p$, so $\E Y_1 = 1/p$. Then for any $\delta > 0$ and any $M \in \N$, 
\begin{equation*}
\p\Big(\Big| \sum_{i=1}^M Y_i - M/p \Big| \geq \delta \frac{M}{p} \Big) \leq 2\exp\Big(-(\delta - \log(1+\delta)) \frac{M}{p} \Big).
\end{equation*}

\item[(ii)] Suppose $Z_1, Z_2, \ldots$ are independent exponential random variables with means $a_1, a_2, \ldots$, and set $a_* = \inf_i a_i$, $\kappa = \sum_i a_i$. Then for any $\delta > 0$ and $M \in \N$, 

\begin{equation*}
\p\Big(\Big| \sum_{i=1}^M Z_i - \kappa \Big| \geq \delta \kappa \Big) \leq 2 \exp\Big(-a_* \kappa (\delta - \log(1+\delta) \Big).
\end{equation*} 
\end{enumerate}
\end{lemma}

For the proof of Lemma \ref{l:nstep1} we need the following result. 

\begin{lemma}
\label{stop12} 
Let $n$ be an integer multiple of $r$ and consider fixed locations at distance $r$ on $\cC_n$ (without loss of generality take them to be multiples of $r$). Let $k$ independent identically distributed lazy symmetric random walks  started from arbitrary locations on $\cC_n$ and stopped on hitting the nearest integer multiple of $r$. Let $T_i$ be the total number of steps taken by the $i^{th}$ walk (including the lazy steps).
Then there exists $C>0$ such that $$\P\left(\sum_{i=1}^k T_i \ge Cr^2 k\right) \le e^{-Ck},$$ where $C$ depends on the laziness parameter, (i.e., the probability of not jumping). 
\end{lemma}

\begin{proof}
Standard simple random walk estimates show that if $\tau$ is the hitting time of $\{0,r\}$ for a lazy simple random walk on $\Z$ (with laziness $p$), then for each $x\in \llbracket 0,r\rrbracket$, we have $\P_x(\tau < Cr^2)\ge  \frac{1}{2}$ where $\P_x,$ denotes the probability measure for the random walk started at $x$ and $C$ depends only on the laziness parameter. It follows then that for any arbitrary location of the $k$ particles, each $T_i$ is stochastically dominated by $r^2G$ where $G$ is a geometric random variable with mean $2$. The statement now follows from part (i) of Lemma \ref{lem2}.
\end{proof}

We are now ready to prove Lemma \ref{l:nstep1}.

\begin{proof}[Proof of Lemma \ref{l:nstep1}]
First notice that the total number of instructions used by the particles ignoring the sleep instructions is the same as the total number of steps taken when the particles do independent lazy symmetric random walks on $\cC_n,$ where the laziness (probability of not jumping) is $\frac{\lambda}{1+\lambda}$ (exactly the probability that an instruction is a sleep instruction).
Now as an easy consequence of \eqref{azuma}, for any  $\e>0$,  the total number of particles in $\eta$ is in $(\mu-\e, \mu+\e)n$ with probability at least $1-e^{-cn}$ for some $c=c(\e,\mu)>0$. We can therefore condition on the number of particles being $m\in (\mu-\e, \mu+\e)n$.  Let $T^*$ denote the total number of steps taken by these particles until the end of phase one. Using Lemma \ref{stop12} with $r=c_0\log n$,we get that $\P(T^*> Cn\log^2 n)\leq e^{-cn}$. 
\end{proof}

\subsection{Phase two: Stabilizing from the sources}
We describe the second phase of our toppling scheme now. Recall that we start with the configuration $\eta^{(1)}$ that is supported on the $\mathbf{Source}$ vertices. Throughout the section we shall assume that $\eta^{(1)}$ satisfies the high probability event described in Lemma \ref{l:rw1}, where $\e$ will be chosen sufficiently small depending on $\mu$ and $\lambda$ later. Recall the intervals $I_{i}=[(i-1)c_0\log n, ic_0\log n]$ for $1\le i \le K$. Observe that the $i^{th}$ source $z_i$ is the midpoint of the interval $I_{i}$. 

As mentioned before, for each $1\le i\le K,$ we start stabilizing particles at $z_i$ sequentially, in any arbitrary manner of toppling, until one of the particles hit the boundary of $I_i$, in which case we term the process a failure. On the contrary, we denote by $\mathcal{S}_{i}$ the event that all the $\eta^{(1)}_{i}$ particles all fall asleep before hitting the boundary of $I_i$: call this event \textbf{Success at source $z_i$}. Clearly on the event that $\cS_{i}$ occurs for all $i$, the system stabilizes. The main step is to show that $\cS_i$ occurs with high enough probability so that one can take a union bound over all intervals $I_i$. Because of the underlying symmetry, we state the following result for a generic interval $[-\frac{r}{2}, \frac{r}{2}]$, where we assume $r$ is even to avoid rounding issues.

\begin{proposition}
\label{p:genericfix}
Consider ${\rm ARW},$ with sleep rate $\lambda,$ on $[-\frac{r}{2},\frac{r}{2}],$ started with $m$ particles at the origin. Let $\e>0$ be such that $\mu+2\e < \frac{\lambda}{1+\lambda}$, and let $m\leq (\mu+\e)r$. Let $\cS$ denote the event that all the particles fall asleep before any particle hits $\{-\frac{r}{2},\frac{r}{2}\}$. Then there exists $c>0$, such that for all $r$ sufficiently large we have 
$\P(\cS^c)\leq e^{-cr}.$
\end{proposition}

Proposition \ref{p:genericfix} is the most technically complicated result that goes into the proof of Theorem \ref{sub}, and the proof will be spread over the next two subsections. Before delving into this proof, we want to complete the remaining steps in the argument proving Theorem \ref{sub}. First we need to prove the following easy lemma. 

\begin{lemma}
\label{l:nstep2}
Consider ${\rm ARW}$ with sleep rate $\lambda,$ started from $\eta^{(1)}$. Let $T^{(2)}=T^{(2)}(\mu,\lambda)$ denote the total number of steps taken by all the particles until stabilization. Then there exists $C_3>0,$ and $\theta>0,$ such that on the event $\cap \cS_{i}$, we have 
$$\P\left (T^{(2)}\geq C_3 n\log ^2n \right)\leq n^{-\theta}$$
for all $n$ sufficiently large. 
\end{lemma}
\begin{proof}
Observe that arguing as in the proof of Lemma \ref{stop12}, taking $r=c_0 \log n,$ on $\cS_i$, the number of steps taken by each particle started from $z_i$ is dominated by $Cc_0^2\log^2 n G$ where $G$ is a geometric random variable with mean 2 and $C$ depends on $\lambda$. The rest of the proof is identical to that of Lemma \ref{stop12} and its application in the proof of Lemma \ref{l:nstep1}. We skip the details. 
\end{proof}
We can now complete the proof of Theorem \ref{sub}.

\begin{proof}[Proof of Theorem \ref{sub}] 
Fix $\mu< \frac{\lambda}{1+\lambda}$ and recall our two phase stabilization procedure. Recall $T^{(1)}$ from Lemma \ref{l:nstep1} and $T^{(2)}$ from Lemma \ref{l:nstep2}, and the stack of instructions $\sI$ from \eqref{stack}. 
Note that our toppling scheme produces a stack of instructions $\sI^*$ obtained from $\sI$, where the sleep instructions which are ignored by our toppling scheme are replaced by $\frak{n}$ instructions (see the definitions before Lemma \ref{l:sleepmon}).
Moreover, given the stack $\sI^*,$ by the Abelian Property (Lemma \ref{abp}), on the event $\cap \cS_i,$ the total number of instructions needed to stabilize is $T^{(1)}+T^{(2)},$ since our toppling scheme uses exactly those many instructions from $\sI^*.$ Finally, using Lemma \ref{l:sleepmon}, the total number of instructions used to stabilize, for the stack  $\sI$, is upper bounded by $T^{(1)}+T^{(2)}.$
Thus it will suffice to show that  
\begin{equation}
\label{e:suffice}
\P(T^{(1)}+T^{(2)}> C_0n\log ^2 n)\leq n^{-b}.
\end{equation}
We first fix $\e>0$ so that the hypothesis of Proposition \ref{p:genericfix} is satisfied. Then fix $c_0, C_2, a$ so that the conclusions of Lemma \ref{l:rw1} (with the same $\e$) and Lemma \ref{l:nstep1} hold. 
Thus, the event
$$\cA=\left \{\sup_{1\le i\le K} |\eta^{(1)}_{i}-\mu c_0\log n| \le \e c_0 \log n\right\}\cap \left\{T^{(1)}(\mu,\lambda)< C_2 n\log^2 n\right \}$$ occurs with probability at least $1-n^{-\theta}$ for some $\theta>0.$

On $\cA,$ by definition, for each $1\le i\le K,$ $\eta^{(1)}_i$ satisfies the hypothesis of Proposition \ref{p:genericfix}, and hence applying the latter with $r=c_0\log n$ and a union bound we get
$$\P(\cap \cS_{i})\geq 1-ne^{-hc_0\log n}\geq 1-n^{-\theta}$$
for some $h>0$, and by choosing $c_0$ sufficiently large the final inequality holds for some $\theta>0$. We can now infer \eqref{e:suffice} for all sufficiently large $n,$  from Lemma \ref{l:nstep2}, choosing $C_0$ sufficiently large compared to $C_2$ and $C_3$ and choosing $b$ sufficiently small. 
\end{proof}

\subsection{Setting the Traps}

It remains to prove Proposition \ref{p:genericfix}. For this proof we shall work with the Diaconis-Fulton representation of ARW on $[-\frac{r}{2},\frac{r}{2}]$ using the stack of instructions $\mathscr{I}$ as explained at the end of Section \ref{fsetup}. 

For the remainder of this section, we shall be in the set-up of Proposition \ref{p:genericfix}. Also without loss of generality we shall assume that the total number of particles at the origin is $m=\mu r$, and we run ARW with sleep rate $\lambda$ where $\mu<\frac{\lambda}{1+\lambda}$. Using Abelian Property (Lemma \ref{abp}), our goal will be to provide a toppling procedure (with some possibly ignored sleep instructions) that, when it succeeds, will lead to a stable configuration before any of the particles reach $\{-\frac{r}{2},\frac{r}{2}\}$. Our job will be finished once we show that the algorithm succeeds with high probability. Next we describe in detail the steps of our algorithm, which employs a variant of the algorithm in \cite[Section 5]{RS12}, along with a more complicated trap setting procedure in the finite setting. We elaborate on the differences from \cite{RS12} and their necessity later, but first we describe the procedure.
\subsubsection{Exploration, and locating the traps} Recall that the number of particles at the origin is $m.$ Let us enumerate the particles $y_1,y_2,\ldots, y_{m}$. The algorithm consists of applying a settling procedure to each particle. This procedure explores $\mathscr{I}$ until it identifies a suitable trap for the particle. The exploration follows the path that the particle would perform if we always toppled the site it occupies, and stops when the trap has been chosen. In the absence of a suitable trap,
we declare the algorithm to have failed. We set two initial barriers $a_0 = -\frac{r}{2}, b_0 =\frac{r}{2}$. We will recursively define barrier processes $a_0<a_1< \ldots$ and $b_0>b_1>\ldots$ that are functions of $\mathscr{I}$. Having defined $a_{i}$ and $b_{i}$, we define $a_{i+1}$ and $b_{i+1}$ as follows:
\begin{figure}[h]
\centering
\includegraphics[scale=.55]{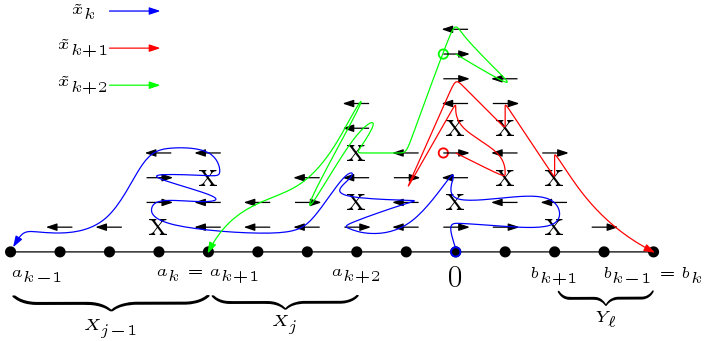}
\caption{This figure is similar to the one appearing in \cite{RS12}. The blue, red and green paths are the exploration trajectories for three consecutive particles labelled $\tilde x_k,\tilde x_{k+1},\tilde x_{k+2}$ respectively, starting at the origin. The $\leftarrow,\rightarrow$ denote jump instructions whereas X denotes a sleep instruction. The first particle is stopped on hitting the barrier $a_{k-1}$ and hence it advances to $a_k$ which is the closest site where the second last instruction was a sleep instruction ignored. The particle now falls asleep at $a_k$ instead of exploring the path beyond $a_k.$ As shown in Lemma \ref{l:gap}, $a_k-a_{k-1}$ is dominated by a Geometric variable of mean $\frac{1+\lambda}{\lambda}.$ However the barrier $b_{k-1}$ stays as it is and is renamed $b_k.$ The second particle hits barrier $b_{k}$ which now advances to $b_{k+1}$ where the particle uses the previously ignored sleep instruction to fall asleep, whereas $a_k$ is renamed $a_{k+1}$. Thus the trap setting scheme proceeds to find traps for each individual particle to fall asleep.}
\label{fig0}
\end{figure}

 Topple the particle $y_{i+1}$ using the previously unexplored instructions in $\mathscr{I}$ until it hits either $a_{i}$ or $b_{i}$: that is, we always topple the site where this particle is currently located. At this stage we ignore all the sleep instructions. Eventually the particle hits either $a_{i}$ or $b_{i}$, call the site hit $q_i$. Let us suppose  $q_i=a_{i}$ and the exploration process hits $q_i$ at step $\tau_i$. We set $b_{i+1}=b_{i}$ and explore the accessed sites backwards from $a_i$, until we reach zero. If we reach a site $v$ where the second to last instruction accessed was a sleep instruction (notice that the last instruction must have been a step towards $a_{i}$) that was ignored, then we set $a_{i+1}=v$, and call $a_{i+1}=\mathbf{Trap}_{i+1}$. If no such site exists we declare the procedure to have failed. Observe that provided we can successfully set the barriers then they are moving towards the origin from both sides. We declare the trap setting scheme a success if $a_{m}<0<b_{m}$. Note that this is a sufficient condition for us to be able to set barriers for all the $m$ particles.

In \cite{RS12}, the argument is based on considering the infinite half line and hence in that setting, it suffices to only consider a single barrier process $a_0 <a_1< \ldots$ In contrast, in our setting, we want the particles to not exit the interval $[-\frac{r}{2},\frac{r}{2}]$,  and hence do not want the situation where a certain random walk hits the barrier process $a_0 <a_1< \ldots,$, after a long excursion  outside the interval of our interest. This creates the necessity to have another barrier process, $\ldots<b_1<b_0,$ which the particle would hit instead on such a journey, preventing its exit from the interval.

\subsubsection{Running the dynamics}
Let us suppose that the realization of $\mathscr{I}$ is such that the trap setting procedure is a success. On this event, let us give a toppling scheme which will utilize the traps to stabilize all the particles before they hit $a_0$ or $b_0$. We topple the particles sequentially. Assuming the particles up to $y_{i-1}$ have been settled, we start the toppling of the particle $y_i$, ignoring all the sleep instructions until the particle hits $\mathbf{Trap}_{i}$ for the last time before hitting $a_{i-1}$ or $b_{i-1}$. Observe that because the instructions used were never used in the in exploration process of the previous particles, the path of this particle is the same as the exploration path, up until it hits $\mathbf{Trap}_{i}$ for the last time. We let the sleep instruction that was the second to last one accessed at $\mathbf{Trap}_{i}$ be executed and this settles the particle $y_{i}$ at $\mathbf{Trap}_{i}$. The key thing to notice here is that all the subsequent instructions accessed by the exploration process (but not in the actual dynamics) are located outside $(a_i,b_i)$, hence the future exploration processes will never try to access them by definition. This implies that the procedure can be continued with all the particles settled at their respective traps if the trap setting procedure succeeds, and additionally the consecutive exploration paths are independent of each other. This fact will be crucial for us when we try to estimate the growth of the barrier processes. 

We summarize the upshot of the toppling procedure and the discussion above in the following lemma. 

\begin{lemma}
\label{l:success}
In the set-up of Proposition \ref{p:genericfix}, suppose the trap setting procedure described above succeeds. Then $\cS$ occurs.
\end{lemma}

\begin{proof}
The proof is a straightforward consequence of Lemmas \ref{l:sleepmon} and \ref{abp}.
\end{proof}

Using Lemma \ref{l:success}, the proof of Proposition \ref{p:genericfix} will now be complete, if we show that the probability that the trap setting procedure fails is exponentially small in $r$. 
As mentioned before, our toppling scheme succeeds if $a_m< 0< b_m$ where $m$ is the total number of particles initially at $0.$
The proof of Proposition \ref{p:genericfix} follows from the next lemma. 

\begin{lemma}
\label{l:trap}
Let $m=\beta r$ be the total number of particles at the origin, and consider the trap setting procedure described above. For $\beta< \frac{\lambda}{1+\lambda}$, there exists $c=c(\beta,\lambda)>0$ such that for all sufficiently large $r$ we have 
$$\P(a_{m}<0<b_{m})\geq 1-e^{-cr}.$$
\end{lemma}

Let us first provide a brief outline of our argument. Observe that at each stage $i$, exactly one of the barriers advances towards the origin. The probability of this being the left one or the right one is equal by symmetry at step 1 and hence equal to $\frac{1}{2}.$ The proof now involves showing that it remains close to $\frac{1}{2}$ throughout. Also we show that the distance a barrier moves at each step has mean $\frac{1+\lambda}{\lambda}$. So the total distance covered by the barriers after $\beta r$ many moves is approximately $\frac{1}{2}\beta r \frac{1+\lambda}{\lambda},$ which is smaller than $\frac{r}{2}$  because of the assumption on $\beta$ and $\lambda$. Since the initial location of the barriers were at $-\frac{r}{2}$ and $\frac{r}{2},$ the above implies that none of the barriers cross the origin. 

We now make the above argument formal. The first lemma we need is the following. A similar observation was already present in \cite{RS12}.

\begin{lemma} 
\label{l:gap}
At the $i^{th}$ stage, the distance of $\mathbf{Trap}_i$ from the barrier $q_{i-1}$ hit by the $i^{th}$ exploration process is dominated by a geometric random variable with mean $1+\frac{1}{\lambda}$ independent of everything else. 
\end{lemma}

\begin{proof} Without loss of generality we assume that $q_{i-1}=a_{i-1}.$ Recall from \eqref{law1} that each instruction $\xi_{(x,j)}$ in $\mathscr{I}$ is a sleep instruction with probability $\frac{\lambda}{1+\lambda}$  and jump instruction otherwise, independent of everything else.  Thus, conditioning on the $i^{th}$ exploration path,  the number of sleep instructions ignored at any site $x\in \cC_n$   between successive jumps at $x$  are i.i.d. random variables distributed as $\rm{Geom}(\frac{1}{1+\lambda}) -1$ (we adopt the standard notation of denote a Geometric random variable with mean $p^{-1}$ by $\mbox{Geom}(p)$). In particular the number of sleep instructions between successive jumps is zero with probability $\frac{1}{1+\lambda}.$ Thus at any site, the probability that there was a  sleep instruction ignored before the last jump instruction is $\frac{\lambda}{1+\lambda}.$  Thus $\mathbf{Trap}_i-a_{i-1}$ is dominated by ${\rm{Geom}}(\frac{1+\lambda}{\lambda})$ variable, independent of everything else. Note that this is not a distributional identity, as $\mathbf{Trap}_i-a_{i-1}$ is bounded above by $r-a_{i-1}.$
\end{proof}

The next step is to show that roughly half of the particles hit the barriers on either side. For $i=1,2,\ldots, m$, let $U_{i}$ denote the indicator that the $i$-th particle exploration process hits the left barrier first, i.e., $q_{i-1}=a_{i}$. Also let $V_{i}=1-U_{i}$. We have the following lemma. 

\begin{lemma}
\label{l:half}
In the above set-up, for each $\delta>0$, there exists $c=c(\delta)>0$ such that 
$$\P\left(\sum_{i=1}^{m} U_{i}\geq \Big(\frac{1}{2}+\delta\Big)m\right)\leq e^{-cr};\quad \P\left(\sum_{i=1}^{m} V_{i}\geq \Big(\frac{1}{2}+\delta\Big)m\right)\leq e^{-cr}.$$ 
\end{lemma}

The proof of Lemma \ref{l:half} is involved and requires a somewhat complicated coupling to a different process; we postpone it to the next subsection. Using this, however, the proof of  Lemma \ref{l:trap} is almost immediate and we complete that part of the argument now. 

\begin{proof}[Proof of Lemma \ref{l:trap}]
Let $X_1,X_2,\ldots$ and $Y_1,Y_2,\ldots$ be two independent sequences of i.i.d.\ $\mbox{Geom}(\frac{\lambda}{1+\lambda})$ variables independent of the sequences $\{U_{i}\}_{i=1}^{m}$ defined above. It follows from Lemma \ref{l:gap} that $a_{m}$ is stochastically dominated by  $-\frac{r}{2}+\sum_{i=1}^{m} X_{i}U_{i}$, and similarly $b_{m}$ stochastically dominates $\frac{r}{2}-\sum_{i=1}^{m}Y_{i}V_{i}$. Now clearly, using Lemma \ref{lem2} it follows that for all $\delta, \delta'>0$ there exists $c=(\delta,\delta')>0$ such that 
$$\P\left(\sum_{i=1}^{(\frac{1}{2}+\delta)m} X_{i} \geq  (1+\delta')\frac{1+\lambda}{\lambda}(\frac{1}{2}+\delta)m \right)\leq e^{-cr}.$$
Choosing $\delta$ and $\delta'$ sufficiently small so that $(1+\delta')\frac{1+\lambda}{\lambda}(\frac{1}{2}+\delta)\beta < \frac{1}{2}$ (this is possible since $\beta< \frac{\lambda}{1+\lambda}$ and $m=\beta r$) it follows using Lemma \ref{l:half} that $\P(a_{m}<0)\geq 1-e^{-cr}$. By symmetry an identical bound holds for $\P(0<b_m)$ and we are done by taking a union bound. 
\end{proof}

\subsection{Coupling with an internal erosion process}
It only remains to prove Lemma \ref{l:half}. As alluded to before to this end we shall use a coupling to a process called internal erosion, (see \cite{LP07} for a nice exposition on the latter). Let $X_1,X_2,\ldots$ and $Y_1,Y_2,\ldots$ be two independent sequences of i.i.d.\ $\mbox{Geom}(\frac{\lambda}{1+\lambda})$ variables. Let $S_{i}=\sum_{j=1}^{i} X_{j}$ and $T_i= \sum_{j=1}^{i} Y_j$ denote the sequence of partial sums.  Let $\tau_1$ (resp.\ $\tau_2$) denote the largest positive integer $i$ such that $S_{i}$ (resp.\ $T_{i}$) is less than $\frac{r}{2}$. Now let $\{Z_{i}\}_{1\leq i< \tau_1}$ (resp.\ $\{W_{i}\}_{1\leq i< \tau_2}$) be a sequence of independent exponential random variables with means $f(i)$ (resp.\ $g(i)$) where $f(i)= \frac{r}{2}-S_{i}$  (resp.\ $g(i)=\frac{r}{2}-T_{i}$). Consider the two following continuous time counting processes: 
$$N^{(1)}(t)=\sup\{n\geq 0: \sum_{i=1}^{n} Z_{i} \leq t\}; \quad N^{(2)}(t)=\sup\{n\geq 0: \sum_{i=1}^{n} W_{i} \leq t\}.$$
Set also $N(t)=N^{(1)}(t)+N^{(2)}(t)$. We shall crucially use the following connection of the above process with the trap setting procedure described in the previous subsection. 

\begin{lemma}
\label{l:coupling}Consider the barrier processes $\{a_i\}, \{b_i\}$ and the internal erosion process described above.
There is a coupling between the two processes satisfying the following: for all $t\geq 0$ such that $N^{(1)}(t)< \tau_1$ and $N^{(2)}(t)< \tau_2$, one has $a_{N(t)}=-\frac{r}{2}+ \sum_{i=1}^{N^{(1)}(t)} X_{i}$ and $b_{N(t)}=\frac{r}{2}-\sum_{i=1}^{N^{(2)}(t)} Y_{i}.$ Moreover, $N^{1}(t)$ is the number of times the barrier $a_0 < a_1<\ldots$ is hit among the first $N(t)$ particles.
\end{lemma}

\begin{proof}
The proof is a consequence of the memoryless property of Exponential variables. Recall from the proof of Lemma \ref{l:gap} that the consecutive non-zero increments of the process $a_0\le a_1\le \ldots,$ are distributed as $X_{1}, X_2,\ldots,$ truncated at certain values which are functions of both the barrier processes. However note that while neither barrier process has reached zero, the issue of truncation does not arise. And hence we can couple the increments of the $\{a_i\}_{i\ge 1}$ exactly to the process $\{X_i\}_{i\ge 1}$ for the first $\tau_1$ increments. Similarly the decrements of the process $\{b_i\}_{i\ge 1}$ can be coupled exactly to the process $\{Y_i\}_{i\ge 1}$ for the first $\tau_2$ decrements (see Figure \ref{fig0} for an illustration.).

Thus to finish the proof of the lemma, we have to argue that the probability of the $j^{th}$ particle hitting the barrier $a_{j-1}$ instead of $b_{j-1}$ is the same as the process $N^{1}(t)$ increasing before $N^{2}(t)$ when $N(t)=j-1$ for any $j$ such that both $N^{1}(t)< \tau_1$ and $N^{2}(t)< \tau_2.$
Note that the probability of the $(N(t)+1)^{th}$ particle hitting the barrier $a_{N(t)}$ as opposed to $b_{N(t)}$ has probability 
\begin{equation}\label{probability}
\frac{b_{N(t)}}{b_{N(t)}-a_{N(t)}}=\frac{g(N^{(2)}(t))}{f(N^{(1)}(t))+g(N^{(2)}(t))}.
\end{equation} 
Note that given the filtration up to time $t$, $N^{1}_t$ increases before $N^{2}(t)$ if and only if $$\sum_{i=1}^{N^{1}(t)+1}Z_i-t\le \sum_{i=1}^{N^{1}(t)+1}W_i-t.$$
Now given the filtration up to time $t,$  using the memoryless property, it follows that  
$$\sum_{i=1}^{N^{1}(t)+1}Z_i-t \mbox{ is distributed as } Z_{N^{1}(t)+1},$$ and similarly $\sum_{i=1}^{N^{2}(t)+1}W_i-t$ is distributed as $W_{N^{2}(t)+1}.$
Thus using the fact that $$\P(Z_{N^{1}(t)+1}<Z_{N^{2}(t)+1})=\frac{g(N^{(2)}(t))}{f(N^{(1)}(t))+g(N^{(2)}(t))},$$ the proof is complete using \eqref{probability}. 
\end{proof}

Using Lemma \ref{l:coupling} to prove Lemma \ref{l:half}, it suffices to prove the following lemma.  Recall that $m=\beta r$ is the total number of particles. 

\begin{lemma}
\label{l:erosion}
Let $\beta < \frac{\lambda}{1+\lambda}$ and let $\delta>0$ be fixed (and sufficiently small  as a function of $\frac{\lambda}{1+\lambda}-\beta$). Let $M_1=(\frac{1}{2}-\delta)\beta r$, $M_2=(\frac{1}{2}+\delta)\beta r$. Let $\mathcal{E}$ denote the event that there exists $t$ such that $\{N^{(1)}(t)\geq M_2, N^{(2)}(t)\leq M_1\}$ or $\{N^{(2)}(t)\geq M_2, N^{(1)}(t)\leq M_1\}$. Then there exists $c>0$ such that $\P(\mathcal{E})\leq e^{-cr}$.   
\end{lemma}

\begin{proof} First observe that, by Lemma \ref{lem2}, $\max(S_{M_2}, T_{M_2})\leq \frac{r}{2}$ (since $\delta$ is sufficiently small) with exponentially small failure probability and hence we have $\min (\tau_1, \tau_2)\geq M_2$ with exponentially small failure probability.  Thus we can safely restrict our analysis to the latter event.  Observe next that it suffices to prove that with exponentially (in $r$) small failure  probability, we have 

\begin{equation}
\label{e:m1m2}
\sum_{i=1}^{M_2} Z_{i}> \sum_{i=1}^{M_1} W_i; \quad \sum_{i=1}^{M_2} Z_{i}> \sum_{i=1}^{M_1} W_i.
\end{equation}

Conditional on the sequences $S_i$ and $T_i$, the concentration estimates in Lemma \ref{lem2} imply that the terms $A_1:=\sum_{i=1}^{M_1}Z_i,$ $A_2:=\sum_{i=1}^{M_1}W_i,$ $A_3:=\sum_{i=1}^{M_2}Z_i , A_4:=\sum_{i=1}^{M_2}Z_i $ are all concentrated near their means $p_1,p_2,p_3,p_4,$ with exponentially small failure probability. The proof is then essentially completed by comparing the means. Note that the means are themselves random (functions of $S_i$ and $T_i$) and hence the last detail is to show that the means themselves are concentrated.

Formally we first observe that $p_1=M_1\frac{r}{2}-\sum_{i=1}^{M_1}i S_{M_1-i+1}$ and similar expressions holds for $p_2,p_3$ and $p_4.$  Note that $\E(p_1)=M_1\frac{r}{2}-\frac{1+\lambda}{\lambda}\frac{M_1^2}{2}+O(M).$
There are several ways to prove concentration of $p_1$ and below we sketch a way to use Lemma \ref{lem2} to achieve this. Note that the latter only allows for sums of geometric variables, whereas we have a linear combination of them.  Since we can afford to be rather crude with our estimates, we use the following decomposition $$p_1=M_1\frac{r}{2}-\sum_{i=1}^{M_1}\sum_{j=1}^{i} S_{j}.$$
Thus  by union bound, after applying Lemma \ref{lem2} to each of the terms of form  $\sum_{j=1}^{i} S_{j},$ it follows that: for all $\e_1$ small enough, there exists $c$ depending on all the parameters except $r,$ such that
\begin{align*}
\P\left(|p_1-\E(p_1)|\ge \e_1 r^2\right)<e^{-cr}.
\end{align*}
Similar analysis allows us to conclude similar bounds as above for $p_2,p_3,p_4$. By choice of $M_1$ and $M_2,$ note that there exists $\e_1$ such that $\E(p_3)-\E(p_1)\ge 4\e_1r^2$ and similarly $\E(p_4)-\E(p_2)\ge 4\e_1r^2.$ 
Thus we see that with probability at least $1-e^{-cr},$ the sequences $S_i,T_i$ are such that
\begin{align*}
p_3-p_1>2\e_1r^2 \,\,\, \mbox{and }
p_4-p_2>2\e_1r^2.
\end{align*}
Moreover, conditioned on the above events, for $j=1,2,3,4,$ Lemma \ref{lem2} implies the following concentration estimates:
\begin{align}
\P\left(|A_j-p_j|\ge \frac{\e_1}{2} r^2\right)>e^{-cr}.
\end{align}

Thus combining the above inequalities and union bound the lemma follows.
\end{proof}

We are now ready to complete the proof of Lemma \ref{l:half}. 

\begin{proof}[Proof of Lemma \ref{l:half}]
We shall use the coupling described in Lemma \ref{l:coupling}. Let $M_1, M_2$ be as in Lemma \ref{l:erosion}. Now by the coupling discussed above and Lemma \ref{l:erosion}, it follows that with exponential (in $r$) failure probability, $M_1$ particles hit both barriers before $M_2$ particles hit any barrier. 
 
Since the total number of of particles is $m\leq M_1+M_2$ it follows that, with exponentially high probability neither $\sum_{i=1}^{m} U_{i}$ nor $\sum_{i=1}^{m} V_{i}$ can exceed $M_2$. Since $\beta < \frac{\lambda}{1+\lambda},$ we can safely ignore the exponentially (in $r$) unlikely event that $M_2\ge \min(\tau_1,\tau_2)$ and hence assume that  coupling in Lemma \ref{l:coupling} does not fail. 
\end{proof}
\begin{remark}\label{sub12} It can be shown that $T_n(\mu,\lambda)$ will be of order $n^3,$ if all the particles (approximately $\mu n$) were initially located at the same site and hence Theorem \ref{sub} relies heavily on the location of the particles in the initial configuration being uniform.
To see this, note that by the above discussion regarding topplings, when a linear in $n$ say $\alpha n,$ number of particles start at the origin, then to stabilize, due to lack of space, at least $\frac{\alpha n}{2}$ particles must move outside an interval of size $\frac{\alpha n}{2}$  centred at the origin. Since a random walk path takes time $\Theta(n^2)$ on average, to exit such an interval, the observation follows. 
\end{remark}

\section{Slow fixation for low sleep rate}  

In this section we prove Theorem \ref{sup}. That is, we prove that for any $\mu>0$ and sufficiently small sleep rate $\lambda$,  ${\rm ARW}(\mu,\lambda)$ on $\cC_n$ takes at least exponentially many steps before reaching the absorbing state with failure probability exponentially small in $n$.

\subsection{The Stabilization Loop}
We shall now describe the toppling scheme outlined in Section \ref{sop} in more detail. Let $\mu\in (0,1)$ be fixed and $\lambda$ be sufficiently small. By Lemma \ref{abp} it suffices to exhibit a sequence of legal topplings with exponentially many steps. We shall show that our procedure satisfies this property with exponentially high probability if $\lambda$ is sufficiently small. 

While running this scheme, particles will switch between two different states, which we call states $X$ and $Y$. Particles in state $X$ follow normal ARW dynamics among themselves as described in Section \ref{fvap}, and can wake up sleeping $Y$-particles. $Y$-particles, on the other hand, do not move, and have no effect on the states of any other particles. Thus a state of the system during this toppling scheme consists of all the particles, each in state $X$ or state $Y$, and each asleep or awake. When a site is toppled, only $X$-particles at that site follow the corresponding stack instructions from $\sI$ (see Section \ref{fvap}). So $Y$-particles only undergo the transition from sleepy to active when an active $X$-particle reaches the same site.

As described in Section \ref{sop}, our toppling procedure will run multiple rounds of what we call stabilization loop. Formally, starting from a particle configuration $\eta$ -- i.e. the values $\eta_t(x)$ for $x \in \cC_n$ -- stabilizing the system in a subset $\cD$ of sites in $\cC_n$ means choosing some particles to be in state $X$ and the rest to be in state $Y$. Then running the particle dynamics described above, only toppling sites inside $\cD$, until all $X$-particles are asleep in $\cD$ or are outside $\cD$. Now the obvious strong Markov property of the above dynamics makes the different stabilization rounds conditionally independent which would be crucial in our calculations.

It will be convenient to identify $\cC_n$ with the interval $[-r,r]$ with $-r$ and $r$ identified, i.e., assume $n = 2r$ for integer $r$. We shall denote the origin by $\mathbf{0}$ and the identified vertex $r=-r$ will be denoted by $\mathbf{r}$. The stabilization steps we run will alternate between taking $\cD = \cC_n \setminus \{\mathbf{0}, \mathbf{r}\}$, $\cD = \cC_n \setminus \{\mathbf{0}\}$ and $\cD = \cC_n \setminus \{\mathbf{r}\}$.

\begin{enumerate}
\item[1.] \textbf{Stabilization Step A:} Stabilize all the particles in $\cC_n\setminus \{\mathbf{0},\mathbf{r}\}$. That is, treat particles in $\cC_n\setminus \{\mathbf{0},\mathbf{r}\}$ as $X$-particles, stopped upon hitting $\{\mathbf{0}, \mathbf{r}\}$, and all other particles as $Y$-particles. So at the end of this procedure all active particles will be at $\mathbf{0}$ or $\mathbf{r}$.

\item[2.] \textbf{Stabilization Step B:} 
Reset the $X$ and $Y$ labels: the particles initially at $\mathbf{0}$ become $X$-particles, and all other particles become $Y$-particles. 
Then stabilize all $X$-particles  with particles stopped at $\mathbf{r}$. With the identification of $\cC_n$ with $[-r,r]$, this step is the same as stabilizing the ARW dynamics in the interior of $[-r,r]$ where the initial particle configuration is supported at the  the center of the interval, a special case of the more general process analyzed later in Lemma \ref{stable12} using results from \cite{BGH15}.

\item[3.]\textbf{Stabilization Step C:} This is identical to the \textbf{Stabilization Step B} above with the roles of $\mathbf{0}$ and $
\mathbf{r}$ 
interchanged.  
\end{enumerate}

The algorithm receives an initial particle configuration $\eta$ on $\cC_n$ drawn from $\mathbb{P}^\mu$ as an input. 
Then we perform  the \textbf{Stabilization Loop}, 
which is \textbf{Stabilization Step A}, followed by  \textbf{Stabilization Step B} and  \textbf{Stabilization Step C}.
We repeat the \textbf{Stabilization Loop} until all of the particles are asleep. 
\medskip

We now state the main lemma about the \textbf{Stabilization Loop} and use it to prove Theorem \ref{sup}.

\begin{lemma}
\label{exploop}
Fix $\mu\in (0,1)$, $\e\in (0,2\mu/5)$ and any particle configuration $\eta$ with at least $(\mu-\e) n$ particles, and suppose at least $\mu n / 2$ particles are active in $\eta$. 
Let $\widetilde{\eta}$ denote the configuration after we have performed the \textbf{Stabilization Loop}.
Then for $\lambda = \lambda(\mu) > 0$ sufficiently small, there exists $c > 0$ such that

\[
\p(\widetilde{\eta} \text{ has less than } \mu n / 2 \text{ active particles}) < e^{-c n}.
\]

\end{lemma}
Using this lemma we now prove Theorem \ref{sup}.
\begin{proof}[Proof of Theorem \ref{sup}]
By the Abelian Property it suffices to demonstrate that with high probability there is a toppling algorithm that does not terminate before exponentially many steps are executed. 
By a Chernoff bound for any $\mu$ and $\e$ the probability that there are at least $(\mu-\e)n$ particles is exponentially close to one.
In the initial stage at least $\mu n /2$ of the particles are awake.
By Lemma \ref{exploop}, the number of consecutive rounds that the \textbf{Stabilization Loop} is performed is at least $e^{\frac{c}{2}r}$ with probability at least $1-e^{-\frac{c}{2}r}$ before all the particles are asleep. Each time the \textbf{Stabilization Loop} is performed, there must be at least one jump or sleep instruction occurring in it.
As $n=2r$ this completes the proof of Theorem \ref{sup}.
\end{proof}

To prove Lemma \ref{exploop} we rely on two results that are proved by adapting the analysis in \cite{BGH15}. 
Our first goal is to show that with exponentially high probability after performing \textbf{Stabilization Step A} there are at least $\mu n/4$ active particles. By definition these active particles are all at $\mathbf{0}$ or $\mathbf{r}$.

\begin{lemma}
\label{expstepA}
Fix $\mu\in (0,1), \e\in (0,2\mu/5)$ and any particle configuration $\eta$ with at least $(\mu-\e) n$ particles of which at least $\mu n/2$ are active. 

Let $\widetilde{\eta}^A$ denote the configuration after we have performed the \textbf{Stabilization Step A}.
Then for $\lambda = \lambda(\mu) > 0$ sufficiently small, there exists $c > 0$ such that

\[
\p\left(\widetilde{\eta}^A \text{ has less than } \mu n / 4 \text{ active particles}\right) < e^{-c n}.
\]

\end{lemma}

Thus after performing \textbf{Stabilization Step A} there are likely to be either 
at least $\mu n/8$ active particles at $\mathbf{0}$ or  
at least $\mu n/8$ active particles at $\mathbf{r}$. 
Suppose there at least $\mu n/8$ active particles at $\mathbf{0}$. Our next result says that with high probability after running \textbf{Stabilization Step B} 
at least $90\%$ of active particles that were at $\mathbf{0}$ are now active particles at $\mathbf{r}$. Also all $Y$ particles are now active.

\begin{lemma}
\label{expstepB}
Fix $\mu > 0$ and any particle configuration $\eta$ with  at least $A \geq \mu n/8$ active particles at $\mathbf{0}$. 

Let $\widetilde{\eta}^B$ denote the configuration after we have performed \textbf{Stabilization Step B}.
Then for $\lambda = \lambda(\mu) > 0$ sufficiently small, there exists $c > 0$ such that

\[
\p\left(\widetilde{\eta}^B \text{ has all $Y$ particles active and at least } .9A \text{ active particles at $\mathbf{r}$}\right) > 1- e^{-c n}.
\]
\end{lemma}

We postpone the proofs of Lemma \ref{expstepA} and Lemma \ref{expstepB} for now and complete the proof of Lemma \ref{exploop}.

\begin{proof}[Proof of Lemma \ref{exploop}] 
After performing \textbf{Stabilization Step A}, by Lemma \ref{expstepA} there are at least $\mu n/4$ active particles at either $\mathbf{0}$ or $\mathbf{r}$ except exponentially small failure probability. On this high probability event, after performing \textbf{Stabilization Step B} then except for exponentially small failure probability there are at least $\mu n/8$ active particles at $\mathbf{r}$ by Lemma \ref{expstepB}. 
(Note that this is true no matter how the active particles were split among $\mathbf{0}$ and $\mathbf{r}$ at the start of \textbf{Stabilization Step B}). On the event that, the high probability event happens at both of these steps, when we start \textbf{Stabilization Step C} there are at least $\mu n/8$ active particles at $\mathbf{r}$. Thus we can apply Lemma \ref{expstepB}, and it follows that, again except for exponentially small failure probability, 
at the end of \textbf{Stabilization Step C} all $Y$ particles are active and at least $90\%$ of the $X$ particles are active. Since, by hypothesis, the total number of particles is at least $(\mu-\e)n$ and $\e< 2\mu/5$, easy algebra shows that on the event that none of the three steps resulted in the failure events of exponentially small probability, the loop results in a particle configuration with at least $\mu n/2$ active particles. The lemma now follows by taking a union bound over the three failure events. (Note that implicitly we are using the obvious strong Markov property of the above dynamics.)
\end{proof}

\subsection{Proving Lemmas  \ref{expstepA} and \ref{expstepB}}
Now we show how to derive Lemmas \ref{expstepA} and \ref{expstepB} from \cite{BGH15}. The following lemma will be the key in both arguments. 

\begin{lemma}
\label{expmoment} 

Fix $\delta_0\in (0,1)$ and $c_0>0$. Consider  ARW with sleep rate $\lambda$ on the interval $[-r,r]$ starting from an initial configuration $\eta$ with at least $\delta_0 r$ active particles. Let $S$ denote the number of sleepy particles in $(-r,r)$ after stabilizing. There exists $C=C(\delta_0,c_0)$ and $\lambda_0=\lambda_0(\delta_0,c_0)>0$ such that for all $\lambda< \lambda_0$ 
$$\P(S \geq c_0 r) \leq e^{-Cr} .$$
for all $r$ sufficiently large.  

\end{lemma}

\begin{proof}
This lemma comes from calculations contained  in \cite{BGH15}. 
Recall the odometer function from \eqref{odo12} and let $h(0)$ denote the odometer at the origin at the end of the stabilization process and let 
$\mathcal{E}=\{h(0)\leq r^{6}\}$. 
We break up $\{S \geq c_0 r\}$ in two parts
$$\{S \geq c_0 r\}\subset (\{S \geq c_0 r\}\cap  \mathcal{E}) \cup \mathcal{E}^c .$$

First we use random walk estimates to prove $\P(\mathcal{E}^c)\leq e^{-Cr}$. This is the same as the argument presented in \cite{BGH15} (Lemma 32) but we provide the short proof for completeness. To start, note that the probability that a lazy random walk started arbitrarily inside $[-r,r]$ does not hit $\{-r,r\}$ in $Kr^{2}$ steps is at most  $e^{-Kc'},$ where $c'$ depends on the laziness parameter. Now there are at most $2r$ particles, each of which moves along an independent, $\frac{\lambda}{1+\lambda}-$lazy, random walk trajectory. So the probability that any of these particles take more than $0.5 r^{5}$ steps before hitting $\{-r,r\}$ is exponentially small. A union bound then implies that the sum of the number of steps taken by all the particles before reaching $\{-r,r\}$ is exponentially unlikely to be more than $r^6$ and hence  we get
$$\P(\mathcal{E})=\P(h(0)\le r^6)\ge 1-e^{-Cr}.$$

Fix $c_1$ such that $c_0/4>c_1>0$. Equation (6.21) in \cite{BGH15} shows for $\lambda$ sufficiently small (and $r$ sufficiently large)
$$\E \left((e^{S_1}+e^{S_2})\mathbf{1}_{\mathcal{E}}\right)\le e^{c_1r}$$ where $S_1$ and $S_2$ denote the number of sleepy particles in $(-r,0]$ and $[0,r)$ respectively.
Thus by Markov's inequality,
$$\P(\{S \geq c_0r\}\cap  \mathcal{E}) \leq
\frac{\E \left((e^{S_1}+e^{S_2})\mathbf{1}_{\mathcal{E}}\right)}{e^{\frac{c_0}{2}r}}  
\leq \frac{2e^{c_1r}}{e^{\frac{c_0}{2}r}}
\leq e^{-Cr}$$ for some $C>0,$ since $S>c_0r$ implies either $S_1$ or $S_2$ is at least $\frac{c_0r}{2}.$
\end{proof}

\begin{proof}[Proof of Lemma \ref{expstepA}]
If there are at most $\mu n/4$ active particles after \textbf{Stabilization Step A} then 
at least $\mu n/4$ particles must have fallen asleep in this step. Thus
either
\begin{enumerate}
\item there were initially 
at least $\mu n/5$ active particles on $(0,r)$ in $\eta$ and there
were at least $\mu n/20$ sleepy particles on $(0,r)$ in $\widetilde{\eta}^A$ or
\item there were initially at least 
at least $\mu n/5$ active particles on $(-r,0)$ in $\eta$ and there
were at least $\mu n/20$ sleepy particles on $(-r,0)$ in $\widetilde{\eta}^A$.
\end{enumerate}
By Lemma \ref{expmoment} both of those events have exponentially small probability. Thus the  lemma follows from an  union bound over the two cases.
\end{proof}

For the proof of Lemma \ref{expstepB} we shall use the following lemma which is immediate from Lemma \ref{expmoment} by taking $c_0$ sufficiently small and we omit the proof.

\begin{lemma}\label{stable12}
For each $\delta_0>0,$ the following holds for $\lambda$ sufficiently small. Consider stabilizing any particle configuration $\eta$ supported on $[-r,r]$ i.e. particles hitting $\{-r,r\}$ are ignored. Call the stabilized system $\eta'$. If $\eta$ has $A$ many active particles with $A \geq \delta_0 r $ then 

$$\P\left(\text{the total number of particles supported on $\{-r,r\}$ in $\eta'$ is at least  $.9A$}\right) \geq 1-e^{-cr}.$$   
\end{lemma}

\begin{proof}[Proof of Lemma \ref{expstepB}]
For this lemma we have to show the two events have exponentially small failure probability: (i) $\widetilde{\eta}^{B}$ has at least $0.9A$ active particles at $\mathbf{r}$, and (ii) all $Y$ particles are active. The first part is immediate from Lemma \ref{stable12}. For (ii), observe the following. 

Each $X$-particle, whenever it moves, follows an independent  random walk trajectory that is stopped at hitting $\{-r,r\}$ (the number of steps in this trajectory that is realized has a complicated dependent structure). By symmetry each of these trajectories are equally likely to end at $-r$ and $r$. Since $A\geq \mu n /8$ and using a  standard Chernoff bound, it follows that with failure probability exponentially small in $n$, no more than $0.6 A$ of these trajectories end at $r$ (also at $-r$). Since by the first part, we know that at least $0.9 A$ of these particles follows their trajectories to hit $\{-r,r\}$, this implies that except for exponentially small failure probability, both $r$ and $-r$ are hit by at least $.3A$ many $X$-particles i.e., $\mathbf{r}$ is hit from both positive and negative side. This implies that every $Y$ particle is hit by an $X$ particle, and hence all $Y$ particles are active in $\widetilde{\eta}^{B}$.  
\end{proof}

We finish with a remark on lower bounding $T_n(\mu,\lambda)$.

\begin{remark}\label{sup12}
An exponential upper bound for $T_n(\mu,\lambda)$ is relatively easy to establish.  
Note that starting from any configuration, ignoring sleep instructions, one can get each particle to a different location, using only polynomial in $n$
many instructions with probability at least $1-e^{-cn},$ since for a random walk on $\cC_n,$ the hitting time for any point is a sub-exponential variable at scale $n^2$.  Once the particles are all located at different sites, with probability at least $(\frac{\lambda}{1+\lambda})^n,$ all of their next instructions are sleep instructions which causes the system to stabilize. 
Thus starting from any configuration, the probability of stabilizing in polynomially many steps is at least  $(\frac{\lambda}{1+\lambda})^n,$ which implies an exponential upper bound on the fixation time, by running independent trials of the above argument until one trial does succeed to stabilize the system.
\end{remark}

\bibliography{arw}
\bibliographystyle{plain}

\end{document}